\documentclass[11pt,reqno,tbtags]{amsart}

\def\newblock{\hskip .11em plus .33em minus .07em}

  \newcommand{\pr}{\mathbf{P}}

\usepackage{amsthm}
\usepackage{amssymb}
\usepackage{amsfonts, mathrsfs}
\usepackage{amsmath}
\usepackage[numbers, sort&compress]{natbib}
\usepackage{graphicx}
\usepackage{vmargin, enumerate}
\usepackage{setspace}
\usepackage{paralist}
\usepackage[pdftex]{hyperref}

\def\qed{\relax\ifmmode\hskip2em \Box\else\unskip\nobreak\hskip1em $\Box$\fi}

\newcommand{\eps}{\varepsilon}

\newcommand{\e}{{\mathbf E}}

\newcommand{\p}[1]{{\mathbf P}\left\{#1\right\}}

\newcommand{\bea}{\begin{eqnarray}}
\newcommand{\eea}{\end{eqnarray}}

\newtheorem{theorem}{Theorem}
\newtheorem{lemma}[theorem]{Lemma}

\newtheorem{dfn}[theorem]{Definition}
\newtheorem{conj}{Conjecture}

\newtheorem{cla}[theorem]{Claim}

\newcommand{\scr}[1]{\mathcal{#1}}

\newcommand{\mass}[1]{\mbox{\em mass}(#1)}

\newcommand{\orr}[1]{\overrightarrow{#1}}

{ \end{list} }

\begin{document}
\title[Connectivity for bridge-addable monotone graph classes]{Connectivity for bridge-addable monotone graph classes \thanks{This research was supported by the Banff International Research station and the McGill Bellairs Research Institute}}

\author{Louigi Addario-Berry}
\address{Department of Mathematics and Statistics, McGill University, 805
  Sherbrooke Street West,  	Montr\'eal, Qu\'ebec, H3A 2K6, Canada}
\email{louigi@math.mcgill.ca}

\author{Colin McDiarmid}
\address{Department of Statistics, 1 South Parks Road, Oxford, OX1 3TG, UK}
\email{cmcd@stats.ox.ac.uk}

\author{Bruce Reed}
\address{School of Computer Science, McGill University, 3480 University Street, Montreal, Quebec, H3A 2A7, Canada}
\email{breed@cs.mcgill.ca}

  \maketitle


\begin{abstract}

A class $\scr{A}$ of labelled graphs  is {\em bridge-addable} if 
if for all graphs $G$ in $\scr{A}$ and all vertices $u$ and $v$ in distinct connected components of $G$, 
the graph obtained by adding an edge between $u$ and $v$ is also in $\scr{A}$; 
the class $\scr{A}$ is {\em monotone} if for all $G \in \scr{A}$ and all subgraphs $H$ of $G$, we have $H \in \scr{A}$. 
We show that for any bridge-addable, monotone class $\scr{A}$ whose elements have vertex set $\{1,\ldots,n\}$, 
the probability that a uniformly random element of $\scr{A}$ is connected is at least $(1-o_n(1)) \, e^{-\frac12}$, 
where $o_n(1)\rightarrow 0$ as $n\rightarrow \infty$.
This establishes the special case of the conjecture of \cite{mcdiarmid06random} when the condition of monotonicity is added.
This result has also been obtained independently by Kang and Panagiotiou (2011). 
\end{abstract}


 \section{Introduction}
  \label{sec.intro}

    Given a class $\scr{A}$ of graphs, we say that $\scr{A}$ is {\em bridge-addable}
    (or {\em weakly addable}) if for all graphs $G$ in $\scr{A}$ and all vertices
    $u$ and $v$ in distinct connected components of $G$, the graph obtained by adding an
    edge between $u$ and $v$ is also in $\scr{A}$.
    The concept of bridge-addability was introduced in 
    McDiarmid, Steger and Welsh~\cite{mcdiarmid05random}
    in the course of studying random planar graphs.
    It was shown there
    that, for a uniformly random element of a finite non-empty
    bridge-addable class $\scr{A}$ of labelled graphs,
    the probability that it is connected is at least~$e^{-1}$.

    As well as the class of planar graphs, other examples of bridge-addable graph classes
    include forests, graphs with tree-width at most $k$,
    graphs embeddable on any fixed surface,
    and more generally any minor-closed graph class with cut-point-free excluded minors; 
    triangle-free graphs,
    and more generally $H$-free graphs for any two-edge-connected 
    graph $H$; and $k$-colourable graphs.

    It is well-known that there are $n^{n-2}$ trees on $n$ labelled vertices, Cayley~\cite{cayley89trees}.
    Together with a result of R\'enyi~\cite{renyi59} 
    (see also Moon~\cite{moon70counting}) that the corresponding number of forests
    is asymptotic to $e^{\frac12} n^{n-2}$, we see that,
    for a uniformly random forest on $n$ labelled vertices, the probability that it is
    connected is asymptotically $e^{-\frac12}$.

    In McDiarmid, Steger and Welsh~\cite{mcdiarmid06random} it was suggested that if connectivity
    is desired then the worst possible example of a bridge-addable graph class is the class of forests.
    More precisely, they conjectured that the lower bound of $e^{-1}$ on the probability of connectedness for
    a bridge-addable graph class can be improved asymptotically to $(1+o(1))e^{-\frac12}$.
    (They assumed also that the graph class was closed under isomorphism.)
    Recently, Balister, Bollob{\'a}s and Gerke~\cite{bbg07,bbg10} took a first step towards proving this conjecture,
    proving an asymptotic lower bound of $e^{-0.7983}$.

    Observe that the examples above of bridge-addable graph classes,
    and many other interesting examples of such graph classes,
    also satisfy the property of being {\em monotone} (decreasing);
    that is, given a graph $G$ in $\scr{A}$, each graph obtained by deleting edges from $G$ is also in $\scr{A}$.
    In this paper we investigate the probability of connectivity of a uniformly random element of
    a monotone bridge-addable graph class; for such graph classes, we
    prove the conjectured lower bound.
    %
    %
    Our method relies upon a reduction to weighted random forests; and in particular to the properties of weighted random trees. 
    

  For several bridge-addable graph classes, including some of those mentioned above,
  the asymptotic probability of connectedness has recently been determined --
  see the last section below.
%
%
%

    Let us now state our theorem.
    A {\em bridge} in a graph $G$ is an edge $e$
    such that $G-e$ has strictly more connected components than $G$.
    We say that a class $\scr{A}$ of graphs is {\em bridge-alterable} if
    for any graph $G$ and bridge $e$ in $G$,
    $G$ is in $\scr{A}$ if and only if $G-e$ is.
    We remark that if $\scr{A}$ is bridge-alterable then it is bridge-addable,
    and if it is both bridge-addable and monotone then it is bridge-alterable.

    %
    %
  \begin{theorem}\label{thm:main}
  For any $\eps>0$ there exists $W_0$ such that, if $W \geq W_0$ and
  $\scr{A}$ is a non-empty bridge-alterable class of graphs on $\{1,\ldots,W\}$,
  and if $\bf{G}$ is a uniformly random element of $\scr{A}$, then
  \begin{equation}
    \label{eq:main}
    \p{\bf{G}\mbox{ is connected}} \geq (1-\eps) e^{-\frac12}.
  \end{equation}
  \end{theorem}
  This result was announced independently by Kang and Panagiotou while the present paper was under (slow) revision
  following referee reports,
  with at that time a weaker form of the above theorem.  Their proof in~\cite{kp2011} 
  starts like our proof here but then proceeds very differently, with a clever induction
  involving merging vertices when counting forests with two components 
  (where the vertices correspond to bridgeless graphs).
  
  In the next section we describe a reduction which allows us to establish Theorem~\ref{thm:main} by
  proving that, for a random forest with an appropriately chosen distribution, full connectivity 
  (the event that the forest is a tree) is almost twice as likely as the event that the forest has precisely two connected components (the precise statement appears as Lemma~\ref{lem:ratio_1_2}, below). 
  %
  In the following two sections we explain how
  Lemma~\ref{lem:ratio_1_2} will yield Theorem~\ref{thm:main}, and then prove
  Lemma~\ref{lem:ratio_1_2}.
  Finally we make some concluding remarks.


  \section{A reduction to weighted forests}
  \label{sec.body}
  In this section we assume that the conditions of Theorem~\ref{thm:main} are satisfied. 
  For any class $\scr{A}$ of graphs, let ${\scr{A}}_n$ denote the set of graphs
  in $\scr{A}$ on the vertex set $\{1,\ldots,n\}$.  
  It is convenient to consider graphs on $W$ vertices.
  We shall prove Theorem \ref{thm:main} by partitioning $\scr{A}_W$ and showing that an inequality
  such as (\ref{eq:main}) holds for a uniformly random element of each block of the partition.
  (This step of our proof is essentially Lemma 2.1 of \cite{bbg07}.)

  \begin{dfn}
  Given a graph $G$, let $b(G)$ be the graph obtained by removing all bridges from $G$.
  We say $G$ and $G'$ are {\em equivalent} if $b(G)=b(G')$, and in this case write $G \sim G'$.
  For a graph $G$, let $[G]$ be the set of graphs $G'$ for which $b(G')=b(G)$.
  \end{dfn}
  It is easily seen that $\sim$ is an equivalence relation on graphs, and
  thus we always have $[G]=[b(G)]$. Furthermore, if $G \in \scr{A}_W$ then as
  $\scr{A}$ is closed under deleting bridges, $b(G) \in \scr{A}_W$,
  and as $\scr{A}$ is bridge-addable, we have $[G] \subseteq \scr{A}_W$.
  It follows that $\scr{A}_W$ can be written as a union of
  some set of disjoint equivalence classes $[G_1],[G_2],\ldots$.
  To prove Theorem \ref{thm:main}, we will in fact prove:
  \begin{cla}
  \label{cla:impliesmain}
  For any graph $G \in \scr{A}_W$, if $H$ is a uniformly random element of $[G]$ then
  \[
  \p{H\mbox{ is connected}} \geq e^{-\frac12+o(1)},
  \]
  where $o(1) \rightarrow 0$ as $W \rightarrow \infty$.
  \end{cla}
  Clearly, Theorem \ref{thm:main} immediately follows from Claim \ref{cla:impliesmain},
  and it thus remains to prove Claim \ref{cla:impliesmain}.
  Fix a bridgeless graph $G$ on vertex set $\{1,\ldots,W\}$
  and let $\scr{B}=[G]$.  Write $C_1,\ldots,C_n$ for the components of $G$, and let $w_i=|V(C_i)|$
   for $i=1,\ldots,n$, so $W = \sum_{i=1}^n w_i$.
   We remark that since the components $C_1,\ldots,C_n$ are bridgeless,
   either $w_i=1$ or $w_i\geq 3$ for all $i\in\{1,\ldots,n\}$ (though we shall not use this fact).
   We denote by $\orr{w}$ the vector $(w_1,\ldots,w_n)$.
  We use $\orr{w}$ to define a probability measure on the set $\scr{F}_n$ of
  forests with vertex set $\{1,\ldots,n\}$.
  Given $F \in \scr{F}_n$, let
\[
\mass{F} = \mbox{\em mass}_{\orr{w}}(F) = \prod_{i=1}^n w_i^{d_F(i)},
\]
  where $d_F(i)$ denotes the degree of vertex $i$ in the forest $F$.
  Also, let $K=\sum_{F \in \scr{F}_n} \mass{F}$, and let ${\bf F}$ be a random element of $\scr{F}_n$ with
$\p{{\bf F}=F}=\mass{F}/K$ for all $F \in \scr{F}_n$. We say ${\bf F}$ is {\em distributed according to}
$\orr{w}$; when we wish to highlight the distribution of ${\bf F}$,
we will sometimes write ${\bf F}_{\orr{w}}$ in place of ${\bf F}$.
For our purposes, the key fact about such a random forest is the
following:
\begin{lemma}\label{eq:hf_con}
  For a uniformly random element ${\bf H}$ of $\scr{B}$,
  \[
    \p{{\bf H}\mbox{ is connected}} = \p{{\bf F}\mbox{ is connected}}.
  \]
\end{lemma}
%
%
\begin{proof}
    We construct a flow from $\scr{B}$ to $\scr{F}_n$ in the following fashion:
    given $G\in \scr{B}$, let $f(G)$ be the graph obtained from $G$ by contracting
    $C_i$ to a single point for each $i=1,\ldots,n$; then $f(G)\in\scr{F}_n$,
    and for each $F \in \scr{F}_n$, the set $f^{-1}(F)$ has cardinality precisely
    $\prod_{i=1}^n w_i^{d_F(i)}$. Since $G \in \scr{B}$ in connected if and only
    if $f(G)$ is connected, it follows that
    \begin{eqnarray}
        \p{{\bf H\mbox{ is connected}}} & = & \frac{|\{G\in\scr{B}:G\mbox{ is connected}\}|}{|\scr{B}|}\nonumber\\
                & = & \frac{\sum_{\{F\in\scr{F}_n:F\mbox{ is connected}\}}|f^{-1}(F)|}{\sum_{F\in\scr{F}_n}|f^{-1}(F)|}\nonumber\\
                & = & \frac{\sum_{\{F\in\scr{F}_n:F\mbox{ is connected}\}}\mass{F}}{K}\nonumber\\
                & = & \p{{\bf F}\mbox{ is connected}}\hfill\nonumber
    \end{eqnarray}
    as required.
\end{proof}
To prove Claim \ref{cla:impliesmain}, it therefore suffices to show that for such a random forest ${\bf F}$,
$\p{{\bf F}\mbox{ is connected}} \geq e^{-\frac12+o(1)}$, where $o(1)$ tends to zero as $W \to \infty$.

For $i=1,\ldots,n$, let $\scr{F}_{n,i}$ be the set of elements of
$\scr{F}_n$ with $i$ components (so $F \in \scr{F}_{n,1}$ precisely if $F$ is connected). For larger $i$ set $\scr{F}_{n,i}=\emptyset$.
It turns out that bounds on $\p{{\bf F}\mbox{ is connected}}$ follow from
bounds on the ratio between $\p{F\in\scr{F}_{n,2}}$ and $\p{F\in\scr{F}_{n,1}}$. More precisely, Claim
\ref{cla:impliesmain} follows from Lemma \ref{eq:hf_con} and the following lemma.
  \begin{lemma} \label{lem:ratio_1_2}
    For all $\eps > 0$, for $W$ sufficiently large, for all
    $\orr{w}=(w_1,\ldots,w_n)$ with $\sum_{j=1}^n w_j=W$,
    \begin{equation} \label{eq:ratio_1_2}
        \p{{\bf F} \in \scr{F}_{n,2}} \leq (1+\eps) \frac12 \ \p{{\bf F} \in \scr{F}_{n,1}}.
    \end{equation}
  \end{lemma}
In Section \ref{sec:uses_lemma} we explain how to use Lemma \ref{lem:ratio_1_2} to prove Claim \ref{cla:impliesmain};
in Section \ref{sec:lem_proof} we prove Lemma \ref{lem:ratio_1_2}.
%
%
%

%
%
\section{Proof of Claim \ref{cla:impliesmain} assuming Lemma \ref{lem:ratio_1_2}}\label{sec:uses_lemma}
%
%
    The proof of Claim \ref{cla:impliesmain} proceeds somewhat differently
    depending on the value of the ratio of $n$ and $W$.
    When $W$ is much larger than $n$, the proof is rather straightforward, and in fact
    does not require Lemma \ref{lem:ratio_1_2} at all, but rather Lemma \ref{lem:smalln2.5} 
    below. In both cases, however, we 
    compare the probability masses of $\scr{F}_{n,i}$ and of $\scr{F}_{n,i+1}$ 
    by double-counting edge-weights in a bipartite graph.
    
      Given a graph $G$, let $c(G)$ be the set of connected components of $G$.
      Given a forest $F \in \scr{F}_n$ and $T\in c(F)$,
      let $w(T)=\sum_{i \in V(T)}w_i$.
      Consider forests $F,F' \in \scr{F}_n$ such that $F'$ can be obtained from $F$ by the addition
      of an edge $e$.
      Writing $T\neq T' \in c(F)$ as shorthand for $\{\{T,T'\}\subseteq c(F):T\neq T'\}$,
      we let
      \begin{equation}
        \varphi(F',F) = \frac{\mass{F'}}{\sum_{T\neq T' \in c(F)} w(T)w(T')}.\label{eq:phi_def}
      \end{equation}
      For all other pairs $F,F'$, we let $\varphi(F',F)=0$.
      We use $\varphi$ in a standard double-counting argument.
      We will use the next preliminary lemma also in Section~\ref{sec:lem_proof}.
\begin{lemma}\label{lem:mass_eq}
   For all positive integers $W$ and all positive integer weight vectors
   $\orr{w}=(w_1,\ldots,w_n)$ with $\sum_{j=1}^n w_j=W$, and for each $i=1,\ldots,n-1$
    \begin{equation} \label{eq:sumflow}
     		\sum_{F'\in\scr{F}_{n,i}}\sum_{F\in\scr{F}_{n,i+1}}\varphi(F',F) =
     		\sum_{F\in \scr{F}_{n,i+1}}\mass{F} = K\cdot\p{{\bf F}\in \scr{F}_{n,i+1}}.
    \end{equation}
      \end{lemma}
\begin{proof}
  Given $i\in\{1,\ldots,n-1\}$ and $F \in \scr{F}_{n,i+1}$,
      if $F'\in \scr{F}_{n,i}$ is obtained from $F$ by the addition of edge $uv$,
      then $\mass{F'}=\mass{F}\cdot w_u \cdot w_v$. We thus have
  \begin{eqnarray*} 
    && \sum_{F' \in \scr{F}_{n,i}} \varphi(F',F)\\
    & = & \left(\sum_{T\neq T' \in c(F)} \sum_{u\in V(T),v\in V(T')} \mass{F}\cdot w_u\cdot w_v \right)
       \cdot \left(\frac{1} {\sum_{T\neq T' \in c(F)}w(T)w(T')} \right)\\
    & = & \frac{\mass{F}}{\sum_{T\neq T' \in c(F)}w(T)w(T')}
       \cdot \left(\sum_{T\neq T' \in c(F)} \sum_{u \in V(T),v\in V(T')} w_u \cdot w_v \right)
                                                 \;\; = \;\; \mass{F}.
      \end{eqnarray*}
    The equation~(\ref{eq:sumflow}) now follows on summing over $F$.

\end{proof}
%
\begin{lemma}\label{lem:smalln2.5}
        For all positive integers $W$ and all positive integer weight vectors
        $\orr{w}=(w_1,\ldots,w_n)$ with $\sum_{j=1}^n w_j=W$, and for each $i=1,\ldots,n-1$
   \begin{equation}\label{eq:smalln2.5}
         \p{{\bf F} \in \scr{F}_{n,i+1}} \leq \frac{\p{{\bf F} \in \scr{F}_{n,i}}(n/W)}{i},
   \end{equation}
\end{lemma}
\begin{proof} 
     Fix $i$ with $1 \leq i \leq n-1$.
        By the definition of $\varphi$, for all $F'\in\scr{F}_{n,i}$ we have
        \begin{equation}  \label{eq:kplus1_to_k}
            \sum_{F\in\scr{F}_{n,i+1}}\varphi(F',F) = \mass{F'}
                \cdot \sum_{e\in E(F')}\frac{1}{\sum_{T\neq T' \in c(F'-e)}w(T)w(T')}.
        \end{equation}
        It is well known that for any set of positive integers $a_1,\ldots,a_{i+1}$ with $\sum_{j=1}^{i+1}a_j=W$,
        \[ 
            \sum_{1 \leq j < k \leq i+1} a_ja_k \geq i(W-i) + {i \choose 2}.
        \] 
    [To see this, if $a_1 \geq a_2 \geq 2$
    then let $a'_1=a_1 +1$, $a'_2=a_2 -1$ and $a'_j=a_j$ for each $j \geq 3$.
    Then with sums as above, $\sum_{j<k} a'_j a'_k - \sum_{j<k} a_j a_k = a'_1a'_2 - a_1 a_2 = -a_1+a_2 -1 <0$.
    Hence the sum is minimised when there are $i$ entries 1 and one entry $W-i$.]
        It follows that, for any $F' \in \scr{F}_{n,i}$ and any $e \in E(F')$, we have
        \begin{equation}\label{eq:smalln2}
            \sum_{T\neq T'\in c(F'-e)} w(T)w(T') \geq i(W-i) + {i \choose 2} \geq i(W-i).
        \end{equation}
    Since, if $F' \in \scr{F}_{n,i}$
    then $F'$ has exactly $n-i$ edges, it follows from (\ref{eq:kplus1_to_k}) and (\ref{eq:smalln2}) that for all
    $F' \in \scr{F}_{n,i}$
    \[ \sum_{F\in\scr{F}_{n,i+1}}\varphi(F',F)
      \leq \mass{F'}\cdot (n-i) \cdot \frac{1}{i(W-i)}
      \leq \mass{F'}\cdot \frac{1}{i}\cdot\frac{n}{W}, \]
    so
    \begin{eqnarray*}
       \sum_{F' \in \scr{F}_{n,i}}\sum_{F\in\scr{F}_{n,i+1}}\varphi(F',F) & \leq & 
        \sum_{F' \in \scr{F}_{n,i}}\mass{F'}\cdot \frac{1}{i}\cdot\frac{n}{W} \\
       &  =   & K\cdot\p{{\bf F}\in \scr{F}_{n,i}}\cdot \frac{1}{i}\cdot\frac{n}{W}, 
    \end{eqnarray*}
    and (\ref{eq:smalln2.5}) follows by combining this last result and Lemma~\ref{lem:mass_eq}.     
\end{proof}
    %
    %
    %
  %
  The case $n/W \leq \frac12$ of Lemma~\ref{lem:ratio_1_2} follows immediately from the above lemma;
  and indeed in this case we may complete the proof directly without using the next two lemmas
  -- see the last sentence of this section. 
%
  To explain why Lemma \ref{lem:ratio_1_2} implies Claim \ref{cla:impliesmain} when $n$ is not much smaller than $W$,
  it turns out to be useful to prove a slightly more general implication.
  
  For each finite non-empty set $V$ of positive integers,
  let $\scr{G}(V)$  be the set of all graphs on the vertex set $V$,
  and let $\scr{G}^k(V)$ be the set of all graphs in $\scr{G}(V)$ with exactly
  $k$ components.
  Also, write $\scr{G}_n$ for $\scr{G}(\{1,\ldots,n\})$, and $\scr{G}_n^k$
  for $\scr{G}^k(\{1,\ldots,n\})$.
  For each positive integer $n$, let $\mu_n$ be a measure on the set of all graphs
  with vertex set a subset of $\{1,\ldots,n\}$,
  which is multiplicative on components
  (that is, if $G$ has components $H_1,\ldots,H_k$, then
  $\mu_n(G)=\prod_{i=1}^k \mu_n(H_i)$).
  \begin{lemma}\label{lem:connect_1}
  Suppose there exist $x > 0$ and integers $n \geq m_0 \geq 1$ such that
  \begin{equation}\label{eq:connect_1_a}
    \mu_n(\scr{G}^2(V)) \leq x \mu_n(\scr{G}^1(V))\qquad \mbox{for all }
    V \subseteq \{1,\ldots,n\} \mbox{ with } |V| \geq m_0.
  \end{equation}
  Let $k$ be a positive integer and suppose that $n \geq km_0.$ Then
  \begin{equation}
    \label{eq:connect_1}
    \mu_n(\scr{G}_n^{k+1})\leq \frac{x}{k}\mu_n(\scr{G}_n^k).
  \end{equation}
  \end{lemma}
  \begin{proof}
  Let $\scr{A}$ be the collection of all sets $\{H_1,\ldots,H_{k-1}\}$ of
  $k-1$ connected graphs such that the vertex sets $V(H_i)$ are pairwise disjoint
  subsets of $\{1,\ldots,n\}$ and
  \[
    \max_{1 \leq i \leq k-1} |V(H_i)| < n-\sum_{i=1}^{k-1}|V(H_i)|.
  \]
  Let $\scr{H}=\{H_1,\ldots,H_{k-1}\}\in \scr{A}$,
  let $V_{\scr{H}}=\{1,\ldots,n\}\setminus\left(\bigcup_{i=1}^{k-1}V(H_i)\right)$,
  and note that
  $|V_{\scr{H}}| > \max_{1\leq i \leq k-1} |V(H_i)|$ and $|V_{\scr{H}}| \geq m_0$.
  For $j=k$ and $k+1$, let $\scr{G}_n^j(\scr{H})$ denote the set of all graphs $G$ in $\scr{G}_n^j$ such that
  $H_1,\ldots,H_{k-1}$ are each components of $G$. Then, letting
  $\alpha=\prod_{i=1}^{k-1} \mu_n(H_i)$,
  by the multiplicativity of $\mu_n$ and by (\ref{eq:connect_1_a}) we have
  \begin{equation} \label{eq:connect_1_pr1}
    \mu_n(\scr{G}_n^{k+1}(\scr{H})) = \alpha\cdot\mu_n(\scr{G}^2(V_{\scr{H}}))
    \leq x\alpha\cdot\mu_n(\scr{G}^1(V_{\scr{H}})) = x\cdot\mu_n(\scr{G}_n^k(\scr{H})).
  \end{equation}
Next, consider any graph $G \in \scr{G}_n^{k+1}$, and suppose that $G$ has
components $G_1,\ldots,G_{k+1}$, where
$|V(G_1)|\leq\ldots\leq|V(G_{k+1})|$.
For each set $\scr{H}$ formed by picking any $k-1$ of the
graphs $G_1,\ldots,G_k$, we have $\scr{H} \in \scr{A}$ and $G \in \scr{G}_n^{k+1}(\scr{H})$.
  It follows that
  \begin{equation} \label{eq:connect_1_pr2}
    k\cdot\mu_n(\scr{G}_n^{k+1}) \leq \sum_{\scr{H}\in\scr{A}} \mu_n(\scr{G}_n^{k+1}(\scr{H})).
  \end{equation}
  Applying (\ref{eq:connect_1_pr1}) to bound the right-hand side
  of~(\ref{eq:connect_1_pr2}), we obtain
  \begin{equation} \label{eq:connect_1_pr3}
    k\cdot\mu_n(\scr{G}_n^{k+1}) \leq x\cdot \sum_{\scr{H}\in\scr{A}} \mu_n(\scr{G}_n^k(\scr{H})).
  \end{equation}
Furthermore, the sets $\{\scr{G}^k_n(\scr{H}):\scr{H} \in \scr{A}\}$ are pairwise disjoint subsets of $\scr{G}_n^k$,
so $\sum_{\scr{H}\in\scr{A}} \mu_n(\scr{G}_n^k(\scr{H})) \leq \mu_n(\scr{G}_n^k)$,
which combined with (\ref{eq:connect_1_pr3}) yields that
\[
k\cdot\mu_n(\scr{G}_n^{k+1}) \leq x\cdot \mu_n(\scr{G}_n^k),
\]
  which completes the proof.
\end{proof}
  Lemma \ref{lem:connect_1} allows us to 
  derive bounds on the ratio between $\p{{\bf F} \in \scr{F}_{n,i+1}}$ and
  $\p{{\bf F} \in \scr{F}_{n,i}}$ for $i>1$.
%
\begin{lemma}\label{lem:monotone_wc}
	Suppose that there exist $0 < \gamma < 1$ and $m_0 > 0$,
	such that for any	positive integer weights $\orr{w}=(w_1,\ldots,w_n)$ with  
	$\sum_{k=1}^n w_k \geq m_0$, 
	$\p{{\bf F}_{\orr{w}}\in\scr{F}_{n,2}}\leq \gamma \p{{\bf F}_{\orr{w}}\in\scr{F}_{n,1}}$. 
	Fix any positive integer $j$. 
	Then for $W$ sufficiently large, for all integers $i$ with $1 \leq i \leq j$ and any positive integer weights 
	$\orr{w}=\{w_1,\ldots,w_n\}$ with $\sum_{k=1}^n w_k = W$, 
	\begin{equation}\label{eq:monoton_wc}
		\p{{\bf F} \in \scr{F}_{n,i+1}} \leq \frac{\gamma\p{{\bf F} \in \scr{F}_{n,i}}}{i}.
	\end{equation}
\end{lemma}

%
\begin{proof}
  Suppose $\gamma$ and $m_0$ satisfy the hypotheses of the lemma, and fix some positive integer $j$.
  Observe first that, by Lemma \ref{lem:smalln2.5}, the inequality~(\ref{eq:monoton_wc})
  holds if $n \leq \gamma W$. We may thus assume that $n > \gamma W$.

Let $n \geq m_0$ and consider any weights $w_1,\ldots,w_n$.   Define $\mu_n(G)$ for each graph
$G$ with vertex set $V \subseteq \{1,\ldots,n\}$ by setting
$\mu_n(G) = \prod_{i \in V}w_i^{d_G(i)}$ if $G$ is a forest and $\mu_n(G)=0$ otherwise.
Then $\mu_n$ is multiplicative on components, and by the hypotheses of the lemma, for each
$V \subseteq \{1,\ldots,n\}$ with $\sum_{i \in V} w_i \geq m_0$ we have
\[
  \mu_n(\scr{G}^2(V)) \leq \gamma \mu_n(\scr{G}^1(V)).
\]
Now we may use Lemma~\ref{lem:connect_1} to obtain
\[
  \mu_n(\scr{G}^{k+1}_n) \leq \frac{\gamma}{k} \mu_n(\scr{G}^{k}_n)
\]
whenever $n \geq km_0$. Since $n \geq km_0$ whenever $W \geq km_0/\gamma$, Lemma \ref{lem:monotone_wc} follows.
\end{proof}
\begin{proof}[Proof of Claim \ref{cla:impliesmain} assuming Lemma \ref{lem:ratio_1_2}]

Fix $\alpha$ with $0 < \alpha < 1$, and
choose $j$ large enough that $2/j! \leq \alpha/2$.
Let $\eps > 0$ be small enough that
$(1-\alpha/2)/(1+\eps)^{j} \geq 1-\alpha$.
  We apply Lemma \ref{lem:monotone_wc} 
  with $\gamma=(1+\eps)\frac12$ 
(Lemma \ref{lem:ratio_1_2} guarantees that there exists $m_0>0$ such 
that the hypotheses of Lemma \ref{lem:monotone_wc} hold with this choice of $m_0$ and $\gamma$).
It follows that for $W$ large enough,
for all $i$ with $1\leq i \leq j$ we have
\[
  \p{{\bf F} \in \scr{F}_{n,i+1}} \leq (1+\eps)^i\frac{ \p{{\bf F} \in \scr{F}_{n,1}}}{2^i i!}.
\]
Furthermore, writing $\kappa(F)$ for the number of connected components of $F$,
\begin{eqnarray}
  1 &=& \sum_{i=0}^{n-1}\p{{\bf F} \in \scr{F}_{n,i+1}} \nonumber\\
&\leq&
  (1+\eps)^{j}\sum_{i=0}^{j-1}\frac{\p{{\bf F}\in \scr{F}_{n,1}}}{2^i i!}
  + \p{\kappa({\bf F}) \geq j+1}.\label{eq:impliesmain1}
\end{eqnarray}
By Lemma \ref{lem:smalln2.5}, for all $i \geq 1$,
\[
    \p{{\bf F} \in \scr{F}_{n,i+1}}\leq \frac{(n/W)^i}{i!} \leq \frac{1}{i!},
\]
from which it follows that for all $k \geq 1$,
\[
\p{\kappa({\bf F}) \geq k+1} \leq \sum_{i \geq k} \frac{1}{i!} \leq \frac{2}{k!}.
\]
Combining the latter equation with (\ref{eq:impliesmain1}) yields that
\[
1 \leq (1+\eps)^{j}e^{\frac12}\p{{\bf F}\in \scr{F}_{n,1}} + \alpha/2,\nonumber
\]
so
\begin{equation}\label{eq:implies_fcon}
  \p{{\bf F}\in \scr{F}_{n,1}} \geq
  \frac{1-\alpha/2}{(1+\eps)^{j}e^{\frac12}} \geq \frac{1-\alpha}{e^{\frac12}}.
\end{equation}
As $\alpha > 0$ was arbitrary, (\ref{eq:implies_fcon}) implies that
$\p{{\bf F}\mbox{ is connected}}\geq e^{-\frac12+o(1)}$ which combined with Lemma
\ref{eq:hf_con} proves Claim \ref{cla:impliesmain}.
\end{proof}
  A simpler version of the above argument lets us deduce directly from
  inequality~(\ref{eq:smalln2.5}) in Lemma~\ref{lem:smalln2.5}
  the non-asymptotic result, that $\p{{\bf F}\mbox{ is connected}}> e^{-n/W}$
  for all positive integers $W$ and all weight vectors
  $\orr{w}=(w_1,\ldots,w_n)$ with $\sum_{j=1}^n w_j=W$.


\section{Proof of Lemma \ref{lem:ratio_1_2}}
\label{sec:lem_proof}
  As already noted, Lemma \ref{lem:ratio_1_2} follows immediately from 
  Lemma \ref{lem:smalln2.5} in the special case when $n \leq \frac12 W$: here we will prove the full result.
  Let $W$ be a positive integer and consider any positive integer weight vector
  $\orr{w}=(w_1,\ldots,w_n)$ with $\sum_i w_i = W$. 
  We may assume that $n \geq 2$. 
  Given a tree $T$ with vertex set $[n]:= \{1,\ldots,n\}$ and an edge $e \in T$,
  we denote by $s(T,e)$ the smaller weight component of $T-e$,
  or the component of $T$ containing vertex $1$ if the components have equal weights.
  We call the components of $T-e$ {\em pendant subtrees of $T$}.
%
%
For $i=1,\ldots,\lfloor W/2 \rfloor$, denote by $c(T,i)$ the quantity $|\{e \in T:w(s(T,e))=i\}|$.

    Recall that $K=\sum_{F \in \scr{F}_n} \mass{F}$, and 
    let $K' = \sum_{T \in \scr{F}_{n,1}} \mass{T}=K\cdot \p{{\bf F}\in\scr{F}_{n,1}}$. 
    Let ${\bf T}$ be a random tree with vertex set $\{1,\ldots,n\}$ and such that
\begin{equation}
  \label{eq:rand_tree}
  \p{{\bf T}=T} = \frac{\mass{T}}{K'}.
\end{equation}
  
  Our proof of Lemma~\ref{lem:ratio_1_2} starts with the identity in Lemma~\ref{lem:fn2_bound} below,
  which expresses the ratio of $\p{{\bf F}\in \scr{F}_{n,2}}$ to $\p{{\bf F}\in \scr{F}_{n,1}}$
  as a weighted sum of the values $\e{c({\bf T},i)}$.
  We then see that we can generate the random tree ${\bf T}$ in a natural way using Pr\"{u}fer codes
  -- see Lemma~\ref{lem.prufer}.  This lets us obtain a good upper bound on the probability that ${\bf T}$ has a pendant
  subtree with a given set of vertices (inequality~(\ref{eqn.pend2}) in Lemma~\ref{lem.pend}); and using~(\ref{eqn.pend2})
  we can upper bound the weighted sum mentioned above, and thus complete the proof.
   
\begin{lemma}   \label{lem:fn2_bound}
    \begin{equation}
      \label{eq:fn2_bound}
      \p{{\bf F}\in \scr{F}_{n,2}} = \p{{\bf F}\in\scr{F}_{n,1}}\cdot\sum_{i=1}^{\lfloor W/2 \rfloor} \frac{\e{c({\bf T},i)}}{i(W-i)}.
    \end{equation}
\end{lemma}
\begin{proof}
  By Lemma~\ref{lem:mass_eq} 
  and the definition of the flow $\varphi$ given in (\ref{eq:phi_def}),
\begin{eqnarray*}
   K \cdot \p{{\bf F}\in\scr{F}_{n,2}}
  & = &
  \sum_{F'\in\scr{F}_{n,1}}\sum_{F\in\scr{F}_{n,2}}\varphi(F',F) \nonumber\\
   &=&
  \sum_{T \in \scr{F}_{n,1}} \mass{T}\cdot\sum_{e \in T} \frac{1}{s(T,e)(W-s(T,e))}\nonumber\\
   & = &
  \sum_{i=1}^{\lfloor W/2 \rfloor} \frac{1}{i(W-i)} \sum_{T \in \scr{F}_{n,1}}
   \mass{T}\cdot c(T,i). 
\end{eqnarray*}
    Also, for each $i=1,\ldots,\lfloor W/2\rfloor$,
\begin{eqnarray*}
  \sum_{T \in \scr{F}_{n,1}} \mass{T}\cdot c(T,i)
   & = &
  K\cdot \sum_{T \in \scr{F}_{n,1}} \p{{\bf T}=T}\cdot\p{{\bf F}\in \scr{F}_{n,1}}\cdot c(T,i)\nonumber\\
   & = &
  K\cdot \p{{\bf F} \in \scr{F}_{n,1}}\cdot\e{c({\bf T},i)}. 
\end{eqnarray*}
    Combining these results proves the lemma.
\end{proof}
  Lemma \ref{lem:fn2_bound} allows us to understand the ratio between
  $\p{{\bf F}\in\scr{F}_{n,2}}$ and $\p{{\bf F}\in\scr{F}_{n,1}}$ by studying the values
  $\e{c({\bf T},i)}$ for $1 \leq i \leq \lfloor W/2 \rfloor$.
  %
  Observe that, for any $k_0>0$, for $W \geq 2k_0$,
\[ \sum_{k=k_0}^{\lfloor W/2 \rfloor} \frac{ \e{c({\bf T},k)}}{k(W-k)} \leq \frac{n-1}{k_0(W-k_0)} \leq \frac{2}{k_0}.\]
  Thus it sufficies to show that for any $\eps>0$, for $W$ sufficiently large we have
\begin{equation} \label{claim.half}
 \sum_{k=1}^{k_0} \frac{\e{c({\bf T},k)}}{k(W-k)} \leq (1+\eps) \cdot \frac12.
\end{equation}  


  



Consider a sequence of iid random variables $Z_1,Z_2,\ldots,Z_{n-2}$ with $\pr(Z_1=i)=\frac{w_i}{W}$ for $i=1,\ldots,n$, 
and let ${\bf Z}=(Z_1,\ldots,Z_{n-2})$. 
Given ${\bf z}=(z_1,\ldots,z_{n-2}) \in [n]^{n-2}$ let $n_i({\bf z})= | \{j: z_j=i\}|$ for each $i=1,\ldots,n$.
  We remind the reader that the {\em Pr\"{u}fer code} construction (see for example the book by West~\cite{west}) 
  gives a bijection $\phi: [n]^{n-2} \to \scr{F}_{n,1}$ such that if $\phi({\bf z})=T$ then
$d_T(i)= n_i({\bf z})+1$ for each~$i$. (Recall that $\sum_j d_{T}(j)-1 = n-2$ for any tree $T$ on $[n]$.)

\begin{lemma} \label{lem.prufer}
The random variables $\phi({\bf Z})$ and ${\bf T}$ are identically distributed. 
\end{lemma}
\begin{proof}
%
Fix any tree $T_0$ on $[n]$.
There is exactly one sequence ${\bf z} \in [n]^{n-2}$ with $\phi({\bf z})=T$, and this sequence must contain 
exactly $d_{T_0}(j)-1$ co-ordinates $j$ for each $j \in [n]$. We thus have 
\[
  \pr(\phi({\bf Z})= T_0) = \pr({\bf Z}={\bf z}) =
 \prod_{j=1}^{n} \left( \frac{w_j}{W} \right)^{d_{T_0}(j)-1}
  = \frac{ \mass{T_0}}{ (\prod_j w_j) \cdot W^{n-2}}. 
\]
  We must then have $K'= (\prod_j w_j) \cdot W^{n-2}$ and, comparing with (\ref{eq:rand_tree}), the result follows. 
\end{proof}
  %

  
  \begin{dfn}
Given $I \subseteq \{1,\ldots,n\}$, let $P_I$ be the event that ${\bf T}$ contains a pendant subtree $T$ with $V(T)=I$, 
and write $w(I) = \sum_{i \in I}w_i$. 
\end{dfn}

\begin{lemma} \label{lem.pend}
  For each $I \subseteq [n]$,
\begin{equation} \label{eqn.pend}
 \pr(P_I)= \left(\frac{w(I)}{W}\right)^{|I|-1} \ \left(1-\frac{w(I)}{W}\right)^{n-|I|-1};
\end{equation}
  and so, for any $\delta>0$ and $w_0>0$, there is a $W_0$ such that
  for each $W \geq W_0$ and each $I \subseteq [n]$ with $w(I) \leq w_0$
\begin{equation} \label{eqn.pend2}
  \pr(P_I) \leq (1+ \delta) \left(\frac{w(I)}{W}\right)^{|I|-1} \ e^{-\frac{nw(I)}{W}}.
\end{equation}
\end{lemma}

\begin{proof}
  Let $I= \{1,\ldots,i\} \subseteq [n]$. According to the Prufer bijection, $I$ is the set of vertices in a pendant subtree of $\phi({\bf z})$
  where ${\bf z}=(z_1,\ldots,z_{n-2})$, if and only if each of $z_1,\ldots,z_{i-1}$ is in $I$
  and none of $z_i,\ldots,z_{n-2}$ is in $I$.  Hence the probability that $I$ is the set of vertices in a pendant
  subtree of $\phi({\bf Z})$ equals the right side of~(\ref{eqn.pend}). 
  This probability depends on the weights of the vertices in $I$ but does not depends on their labels,
  so~(\ref{eqn.pend}) must hold for each set $I \subseteq [n]$.
  Further, the inequality $1-x \leq e^{-x}$ gives  
\[   \left(1-\frac{w(I)}{W}\right)^{n-|I|-1} \leq  e^{-\frac{nw(I)}{W}} \cdot e^{\frac{w(I)(|I|+1)}{W}},\]
  and so~(\ref{eqn.pend}) yields~(\ref{eqn.pend2}).
\end{proof}


  Now a little manipulation of generating functions will allow us to use~(\ref{eqn.pend2}) to upper bound the sum on the right side of~(\ref{eq:fn2_bound}), and thus complete the proof.
  
  Let $F(y)= \sum_{i=1}^{n} y^{w_i}$ and let $G(y)=\frac1{n} F(y)$.
  Next, let $X$ be uniformly distributed on $[n]$, and let $Y=w_X$, so that 
  $Y$ has probability generating function $G(y)$.
  Let $Y_1,Y_2,\ldots$ be independent, each distributed like $Y$, and for $i \ge 1$ let $Z_i=Y_1+\cdots+Y_i$.
  Then for each positive integer $k$
\[ [y^k] G(y)^i = \pr(Z_i=k) := p_i(k). \]
  Thus
\begin{eqnarray*}
  \sum_{I\subseteq [n], w(I)=k} x^{|I|}
  &=&
  [y^k] \sum_{I\subseteq [n]} x^{|I|} y^{w(I)}
  \;\; = \;\; [y^k] \sum_{i \geq 1} x^i \sum_{I\subseteq [n], |I|= i} y^{w(I)}\\
   &=&
  [y^k] \sum_{i \geq 1} \frac{x^i}{i!} \sum_{x_1,\ldots,x_i \in [n] \mbox{{\small distinct}}} y^{\sum_j w_{x_j}}\\
   & \leq &
  [y^k] \sum_{i \geq 1} \frac{x^i}{i!} \sum_{x_1,\ldots,x_i \in [n]} y^{\sum_j w_{x_j}}
   \;\; = \;\;
  [y^k] \sum_{i \geq 1} \frac{x^i}{i!} \ F(y)^i \\
  & = & \sum_{i \geq 1} \frac{(nx)^i}{i!} [y^k] G(y)^i
  \;\; = \;\;
  \sum_{i \geq 1} \frac{(nx)^i}{i!} \ p_i(k).
\end{eqnarray*}
  %
  (Since $p_i(k)=0$ for $i>k$ we could replace $i \geq 1$ in the sum above by $1 \leq i \leq k$.)
  
  We will use this result with $x= \frac{k}{W}$.  Let $\alpha = \frac{n}{W}$,
  so that $nx=\alpha k$.  
  By the last inequality and~(\ref{eqn.pend2}) in Lemma~\ref{lem.pend}, 
  for each $\delta>0$ and each $k$, for $W$ sufficiently large
  (both for~(\ref{eqn.pend2}) and so that $\left(1-\frac{k}{W}\right)^{-1} \leq (1+\delta)$)
  we have
\begin{eqnarray*}
  \frac{\e{c({\bf T},k)}}{k(W-k)}
  &=& \frac{1}{k(W-k)}  \sum_{I \subseteq [n], w(I)=k} \pr(P_I)\\
  & \leq & (1+\delta) \left(1-\frac{k}{W}\right)^{-1} \ k^{-2} e^{-\frac{nk}{W}} \sum_{I \subseteq [n], w(I)=k} \left(\frac{k}{W}\right)^{|I|}\\
  & \leq & (1+\delta)^2 k^{-2} e^{-\alpha k}  \sum_{i \geq 1} \frac{(\alpha k)^i}{i!} p_i(k).
\end{eqnarray*}
  Now fix $\eps > 0$ and $k_0 > 0$. By the preceding inequality, for $W$ sufficiently large,
  for all $k \le k_0$ we have 
\[  (1+\eps)^{-1} \cdot \frac{\e{c({\bf T},k)}}{k(W-k)} \leq k^{-2} e^{-\alpha k} 
   \sum_{i \geq 1} \frac{(\alpha k)^i}{i!} p_i(k),\]
 %
 and thus
\begin{eqnarray*}
  (1+\eps)^{-1} \cdot \sum_{k = 1}^{k_0} \frac{\e{c({\bf T},k)}}{k(W-k)}
  & \leq & \sum_{k = 1}^{k_0} k^{-2} e^{-\alpha k}  \sum_{i \geq 1} \frac{(\alpha k)^i}{i!} p_i(k) \\ 
  & = &  \sum_{1 \leq i \leq k_0} \frac{\alpha^i}{i!} \sum_{k=i}^{k_0} k^{i-2} e^{-\alpha k} p_i(k)\\
  & \leq & \sum_{i \geq 1} \frac{\alpha^i}{i!} i^{-2} \sum_{k\geq i} k^{i} e^{-\alpha k} p_i(k) \\
  &=&  \sum_{i \geq 1} \frac{\alpha^i}{i!}i^{-2} \e [Z_i^{i} e^{-\alpha Z_i}]. 
\end{eqnarray*}
  For $i=1,2,\ldots$ let $f_i(x)= x^i e^{-\alpha x}$ for $x>0$.
  Then $f'(x) = x^{i-1} e^{-\alpha x}(i - \alpha x)$.  Thus $f_i(x)$ takes its maximum value at $x=i/\alpha$,
  and its maximum value is $\left( \frac{i}{\alpha e} \right)^i$.  Hence for $W$ sufficiently large
\begin{eqnarray*}
  (1+\eps)^{-1} \cdot \sum_{k = 1}^{k_0} \frac{\e{c({\bf T},k)}}{k(W-k)}
   & \leq &  
  \sum_{i \geq 1} \frac{\alpha^i}{i!} i^{-2} \e [ f_i (Z_i) ]\\
   & \leq &  
  \sum_{i \geq 1} \frac{\alpha^i}{i!} i^{-2} (\frac{i}{\alpha e})^i \;\;
  = \;\;  
   \sum_{i \geq 1} \frac{i^{i-2}}{i! e^i} \;\; = \;\;  \frac12. 
\end{eqnarray*}
The final equality above can be found, for example, in \citep[page 109]{bollobas01random}. 
 We have now established~(\ref{claim.half}); this completes the proof of Lemma~\ref{lem:ratio_1_2}, and thus of
 Theorem~\ref{thm:main}.


\section{Concluding remarks}

  We have proved the conjecture from McDiarmid, Steger and Welsh~\cite{mcdiarmid06random} that the class of forests is asymptotically 
  the worst possible example of a bridge-addable graph class, from the point of view of connectivity,
  but only in the special case when the graph class is bridge-alterable, and so also when the class is monotone
  as well as bridge-addable. 
  Recently, a substantial amount of work has gone into counting the number of random graphs in a variety of graph
  classes that are bridge-addable; in some cases, this has also led to precise estimates on the probability of connectedness.
  \citet{gn09a} have shown that for a uniformly random planar graph, the probability of connectedness is approximately 0.963253  
  (correct to 6 decimal places, as are all the figures in this paragraph).
  The same results hold for the class of graphs embeddable on any fixed surface, Bender and Gao~\citep{bg11},
  Chapuy, Fusy, Gim\'enez, Mohar and Noy~\cite{cfgmn11}. 
  Similarly, Bodirsky, Gim\'enez, Kang and Noy~\cite{bodirsky05number,bgkn07}
  have shown that for series-parallel graphs and outerplanar graphs, the probabilities 
  of connectedness are approximately 0.889038 and 0.862082, respectively;
  and Gerke, Gim\'enez, Noy and Wei{\ss}l~\cite{ggnw07} have shown that 
  for random $K_{3,3}$-minor-free graphs, this probability is approximately 0.963262.
  For further related results 
  see \cite{bnw09,fp11,gimenez07graph,gn09b}.  
  All these results 
  are for monotone graph classes.

The original conjecture from~\cite{mcdiarmid06random} still is open.
We venture the following, stronger conjecture.
\begin{conj} \label{conj:precise}
 For any $n$ and any non-empty bridge-addable set $\scr{A}$ of graphs
on $\{1,\ldots,n\}$, if
 $\bf{G}$ is a uniformly random element of $\scr{A}$ and
 $\bf{F}$ is a uniformly random element of $\scr{F}_n$, then
\begin{equation}
   \label{eq:precise}
   \p{\bf{G}\mbox{ is connected}} \geq \p{\bf{F}\mbox{ is connected}}.
\end{equation}
\end{conj}
This conjecture would of course yield the original conjecture. 


\bibliographystyle{plain}

\end{document}